\documentclass[12pt]{amsart}
\usepackage{amsmath, amsthm, amsfonts, amssymb}

\newtheorem{theorem}{Theorem}[section]
\newtheorem{lemma}[theorem]{Lemma}

\newtheorem{corollary}[theorem]{Corollary}
\theoremstyle{definition}

\newtheorem{example}[theorem]{Example}

\theoremstyle{remark}

\numberwithin{equation}{section}
\newtheorem{question}[theorem]{Question}

\newcommand{\mC}{\ensuremath{\mathbb{C}}}
\newcommand{\mD}{\ensuremath{\mathbb{D}}}

\newcommand{\mN}{\ensuremath{\mathbb{N}}}

 \allowdisplaybreaks
 
\begin{document}

\title{Normality Criteria Concerning Composite Meromorphic Functions}

\author[V. Singh]{Virender Singh}
\address{Department of Mathematics, University of Jammu,
Jammu-180 006, India}
\email{virendersingh2323@gmail.com }

\author[K. S. Charak]{ Kuldeep Singh Charak}
\address{Department of Mathematics, University of Jammu,
Jammu-180 006, India}
\email{kscharak7@rediffmail.com}

\begin{abstract}
In this paper, we prove normality criteria for families of meromorphic functions involving sharing of a holomorphic function by a certain class of differential polynomials. Results in this paper extend the works of different authors carried out in recent years.
\end{abstract}

\renewcommand{\thefootnote}{\fnsymbol{footnote}}
\footnotetext{The research work of the first author is supported by the CSIR India.}

\maketitle

\section{\textbf{Introduction and Main Results}}
Let $\mathcal{F}$ be a family of meromorphic functions in a domain $D$ with all zeros of multiplicity at least $k$, $P_w\equiv P(z,w):=\prod\limits_{l=1}^{n}(w-a_l(z))$ be a polynomial with holomorphic functions $a_l(z)(1\leq l\leq n)$ as the coefficients, $\alpha(z)$ be a holomorphic function on $D$ such that  $P(z_0,w)-\alpha(z_0)$ has at least two distinct zeros for every $z_0 \in D$; and 
$$M[f]:=f^{n_0} (f')^{n_1} \cdots(f^{(k)} )^{n_k}, \ (k\geq 1),$$
be a differential monomial of $f\in\mathcal{F}$ with degree $\gamma_M:=\sum\limits_{j=0}^{k}n_j$ and weight $\Gamma_M:=\sum\limits_{j=0}^{k}(j+1)n_j$, where $n_0\geq 1$ and $n_j(1\leq j\leq k)$ are non-negative integers such that $\sum\limits_{j=1}^{k} n_j \geq 1$.\\

Regarding the normality of $\mathcal{F}$, we consider the following question:

\medskip

\begin{question}\label{question1} If for every $f,g\in \mathcal{F}$, $P_w\circ M[f]$ and $P_w \circ M[g]$ share $\alpha(z)$ IM, is it true that $\mathcal{F}$ is a normal family?
\end{question}

\medskip

The motivation for proposing this question is by the works of Fang and Yuan \cite{fang-1}, Bergweiler \cite{bergweiler-1}, Yuan, Li and Xiao \cite{yuan-1} and Yuan, Xiao and Wu (\cite{yuan-2},\cite{xiao-1}). In fact, Fang and Yuan \cite[Theorem 4, p.323]{fang-1} proved:
{\it If $\mathcal{G}$ is a family of holomorphic functions in a domain $D$, $P(z)$ is a polynomial of degree at least 2, $\alpha(z)$ is  a holomorphic function such that $P(z)-\alpha (z)$ has at least two distinct zeros, and if $P \circ f(z)\neq\alpha(z)$ for each $f\in\mathcal{G}$, then $\mathcal{G}$ is normal in $D$.} 
Bergweiler \cite[Theorem 4, p.654]{bergweiler-1} generalised this result and proved: {\it If $\mathcal{G}$ is a family of holomorphic functions in a domain $D$, $R(z)$ is a rational function of degree at least 2, $\alpha(z)$ is  a non constant meromorphic function, and if $R \circ f(z)\neq\alpha(z)$ for each $f\in\mathcal{G}$, then $\mathcal{G}$ is normal in $D$.}
Yuan, Li and Xiao \cite{yuan-1} further improved this result and proved: {\it If  $\alpha(z)$ is a non constant meromorphic function, $R(z)$ is a rational function of degree at least $2$, and $R \circ f$ and $R \circ g$ share $\alpha (z)$ IM  for all $f(z), \ g(z)\in\mathcal{G}$, then $\mathcal {G}$ is normal in $D$ if one of the following conditions holds:
\begin{enumerate}
\item $R(z)-\alpha(z_0)$ has at least two distinct zeros or poles for any $z_0\in D$
\item There exists $z_0 \in D$ such that $R(z)-\alpha(z_0):=P(z)/Q(z)$ has only one distinct zero (or pole) $\beta_0$ and suppose that the multiplicities $l$ and $k$ of zeros of $f(z)-\beta_0$ and $\alpha(z)-\alpha(z_0)$ at $z_0$, respectively, satisfy $k\neq lp$ (or $k \neq lq)$, for each $f\in \mathcal{F}$, where $P$ and $Q$ are two co-prime polynomials with degree $p$ and $q$ respectively.
\end{enumerate}
}

\medskip

And in 2011, Yuan, Xiao and Wu (\cite{yuan-2}, \cite{xiao-1}) extended this result and proved:\\

\textbf{Theorem A \cite[Theorem 1.4, p.436]{yuan-2}.} Let $\alpha(z)$ be a holomorphic function, $\mathcal{G}$ be a family of meromorphic functions in a domain $D$ and $P(z)$ be a polynomial of degree at least 3. If $P \circ f(z)$ and $P \circ g(z)$ share $\alpha(z)$ IM for each pair $f, \ g\in\mathcal{G}$ and one of the following conditions holds:
\begin{enumerate}
\item $P(z)-\alpha(z_0)$ has at least three distinct zeros for any $z_0\in D$
\item There exists $z_0 \in D$ such that $P(z)-\alpha(z_0)$ has at most two distinct zeros and $\alpha(z)$ is non constant. Assume that $\beta_0$ is the zero of $P(z)-\alpha(z_0)$ with multiplicity $p$ and that the multiplicities $l$ and $k$ of zeros of $f(z)-\beta_0$ and $\alpha(z)-\alpha(z_0)$ at $z_0$, respectively, satisfy $k\neq lp$, for all $f\in \mathcal{G}$,
\end{enumerate}
then $\mathcal{G}$ is normal in $D$.

\medskip

\textbf{Theorem B \cite[Theorem 1.2, p.2]{xiao-1}.} Let $\alpha(z)$ be a holomorphic function and $\mathcal{G}$ a family of holomorphic functions in a domain $D$. If $P_w \circ f(z)$ and $P_w \circ g(z)$ share $\alpha(z)$ IM for each pair $f, \ g\in\mathcal{G}$ and one of the following conditions holds:
\begin{enumerate}
	\item  $P(z_0,w)-\alpha(z_0)$ has at least two distinct zeros for any $z_0\in D$
  \item There exists $z_0 \in D$ such that $P(z_0,w)-\alpha(z_0)$ has only one distinct zero and $\alpha(z)$ is non constant. Assume that $\beta_0$ is the zero of $P(z_0,w)-\alpha(z_0)$ and that the multiplicities $l$ and $k$ of zeros of $f(z)-\beta_0$ and $\alpha(z)-\alpha(z_0)$ at $z_0$, respectively, satisfy $k\neq lp$, for all $f\in \mathcal{G}$,
\end{enumerate}
then $\mathcal{G}$ is normal in $D$.

\medskip

In this paper we answer Question \ref{question1} as
\begin{theorem}\label{theorem1} If for every $f,g\in \mathcal{F}$, $P_w \circ M[f]$ and $P_w \circ M[g]$ share $\alpha(z)$ IM, then $\mathcal{F}$ is a normal family in $D$.
\end{theorem}

\begin{example} Consider the family $\mathcal{F}$=$\{$$f_n :n\in\mN\}$, where $$f_n (z)=nz^{k-1};k\geq2$$ in the unit disk $\mD$, $P(z,w)=(w+e^z)(w-e^z)$ and $\alpha(z)\equiv 0$. Then, for every $f_n,f_m\in \mathcal{F}$, $P_w \circ  M[f_n](z)$ and $P_w \circ M[f_m](z)$ share $\alpha(z)$ IM. However, the family $\mathcal{F}$ is not normal in $\mD$. Thus, the condition that every $f\in\mathcal{F}$ has zeros of multiplicity at least $k$ is essential in Theorem \ref{theorem1}.
\end{example}

\begin{example} Consider the family $\mathcal{F}$=$\{$$f_n :n\in\mN\}$, where $$f_n (z)=\frac{1}{nz}$$ in the unit disk $\mD$, $P(z,w)=(w+i ze^z)(w-i ze^z)$ and $\alpha(z)=z^2 e^{2z}$. Then
$$P_w \circ M[f_n](z)=\frac{\prod\limits_{r=1}^{k}{{(r!)^{2n_r}}}}{n^{2 \gamma_M}z^{2 \Gamma_M}} + z^2 e^{2z} $$
Clearly, for every $f_n,f_m\in \mathcal{F}$, $P_w \circ M[f_n](z)$ and $P_w \circ M[f_m](z)$ share $\alpha(z)$ IM. However, the family $\mathcal{F}$ is not normal in $\mD$. Thus, the condition that $P(z_0,w)-\alpha(z_0)$ has at least two distinct zeros for any $z_0\in D$ is essential in Theorem \ref{theorem1}.
\end{example}

Let $S=\{a_i(z)/i=1,2,...,n\}$ be the set of holomorphic functions. Then we say that two meromorphic functions $f$ and $g$ share  the set $S$ IM if $\overline{E}_f (S)=\overline{E}_g(S)$, where $\overline{E}_{\phi}(S)=\{z\in D / \prod\limits_{i=1}^{n}(\phi (z)-a_i(z))=0\}$. In that case we write $f(z)\in S \Leftrightarrow g(z)\in S$.

 From Theorem \ref{theorem1}, we immediately obtain the following corollary by setting $\alpha(z)\equiv0$ and $P(z,w)$ a polynomial in variable $w$ that vanishes exactly on a finite set of holomorphic functions. This result can be seen as an extension of a result due to Fang and Zalcman \cite{fang-2} to a general class of differential monomials.

\begin{corollary}\label{corollary1}  Let $S$ be a finite set of holomorphic functions with at least 2 elements. If, for every $f,g\in \mathcal{F}$ and for $z\in D$, $$M[f](z) \in S\Leftrightarrow M[g](z) \in S,$$ then $\mathcal{F}$ is normal in $D$.
\end{corollary}

Further, one can ask what can be said about normality of family $\mathcal{F}$ if set $S$ in Corollary \ref{corollary1} has cardinality equal to one. In this direction, we prove the following result which in turn extends the results of Yunbo and Zongsheng \cite{yunbo-1}, Meng and Hu \cite{meng-1}, Charak and Sharma \cite{charak-1} and Sun \cite{sun-1}.

\begin{theorem}\label{theorem2} Let $n_0,n_1,n_2,...,n_k,k$ be the nonnegative integers such that $n_0\geq2$, $n_k\geq1$ and $k\geq1$. Let $\mathcal{G}$ be a family of meromorphic functions in a domain $D$ and $\alpha(z)\not\equiv0$ be a polynomial of degree $p$. Suppose that each $f\in\mathcal{G}$ has zeros of multiplicity at least $k+p+1$ and poles of multiplicity at least $p+1$. If, for every $f,g\in \mathcal{G}$, 
$M[f](z)$ and $M[g](z)$ share $\alpha(z)$ IM, then $\mathcal{G}$ is normal in $D$.
\end{theorem}

In \cite[Theorem 4, p.323]{fang-1}, for any $f\in \mathcal{G}$, $P \circ f(z)\neq\alpha(z)$ for every $z\in D$, then it is natural to investigate the case when $P \circ f(z)- \alpha(z)$ has zeros. In this direction, we have the following results:

 \begin{theorem}\label{theorem3} If for every $f\in \mathcal{F}$, $P_w \circ M[f](z)-\alpha(z)$ has at most one zero, then $\mathcal{F}$ is normal in $D$.
\end{theorem}

 \begin{theorem}\label{theorem4}  If for every $f\in \mathcal{F}$, $P_w \circ M[f](z)=\alpha(z)$ implies $\left|f(z)\right|\geq M$, for some $M>0$, then $\mathcal{F}$ is normal in $D$.
\end{theorem}

Based on ideas from \cite{charak-1}, one may extend Theorem \ref{theorem1} under the weaker hypothesis of partial sharing of holomorphic function in the following way.

\begin{theorem}\label{theorem5} If for every $f\in \mathcal{F}$, there exist $g\in\mathcal{F}$ such that $P_w \circ M[f](z)$ share $\alpha(z)$ partially with $P_w \circ M[g](z)$, then $\mathcal{F}$ is normal in $D$ provided $P_w \circ M[g](z)\not\equiv \alpha(z) .$
\end{theorem}

\section{\textbf{Proofs of Main Results}}
Besides Zalcman's lemma \cite[p.216]{zalcman-1}, the proofs of our main results rely on the following value distribution results:

\begin{lemma}\label{lemma1} Let $n_0,n_1,n_2,...,n_k,k,p$ be the non-negative integers with $n_0\geq2$, $\sum\limits_{j=1}^{k} n_j \geq 1$ and $k\geq1$. Let $\alpha(z)\not\equiv0$ be a polynomial of degree $p$ and $f$ be a non constant rational function having zeros of multiplicity at least $k+p$ and poles of multiplicity at least $p+1$. Then $f^{n_0} (f')^{n_1} \cdots(f^{(k)} )^{n_k}-\alpha(z)$ has at least two distinct zeros.
\end{lemma}
\begin{proof}Let $\Psi:=f^{n_0} (f')^{n_1} \cdots(f^{(k)} )^{n_k}$. Suppose on contrary that $\Psi(z)-\alpha(z)$ has at most one zero. We distinguish the following  cases:\\
\textbf{Case I}: When $f$ is a non constant polynomial. Since $f$ has only zeros of multiplicity at least $k+p$ and $\alpha(z)$ is a polynomial of degree $p$, we can see that $\Psi(z)-\alpha(z)$ has at least one zero. We set 
\begin{equation}\Psi(z)-\alpha(z)=a(z-z_0)^m,\label{49}
\end{equation}
where $a$ is a non zero constant and $m>p+1$ is a positive integer. Then
 $$\Psi^{(p+1)}(z)=am(m-1)(m-2)\cdots(m-p)(z-z_0)^{m-p-1}$$
implies that $z_0$ is the only zero of $\Psi^{(p+1)} (z)$. Since $f$ is a non constant polynomial, it follows that $z_0$ is a zero of $f(z)$ and hence a zero of $\Psi(z)$ of multiplicity at least $\gamma_\Psi(k+p+1)-\Gamma_\Psi$. Thus $\Psi^{(p)}(z_0)=0$. But, from (\ref{49}), we have $\Psi^{(p)}(z_0)=\alpha^{(p)}(z_0)\neq0$, a contradiction. \\

\textbf{Case II}: When $f$ is a rational but not a polynomial. We set 
\begin{align} f(z)=A\frac{\prod\limits_{i=1}^{s}{(z-{a_{i}})^{m_i}}}{\prod\limits_{j=1}^{t}{(z-{b_{j}})^{l_j}}},\label{55}
\end{align}
where $A$ is a non-zero constant, $m_i\geq k+p(i=1,2,...,s)$ and $l_j\geq p+1 (j=1,2,...,t)$.\\
Put $$M=\sum\limits_{i=1}^{s} m_i$$ and $$N=\sum\limits_{j=1}^{t}l_j.$$
Then $M\geq (k+p)s$ and $N\geq (p+1)t$.\\
Now,
\begin{align} \Psi(z)=A^{\gamma_\Psi} \frac{\prod\limits_{i=1}^{s}{(z-{a_{i}})^{(m_i+1)\gamma_\Psi-\Gamma_\Psi}}}{\prod\limits_{j=1}^{t}{(z-{b_{j}})^{(l_j-1)\gamma_\Psi+\Gamma_\Psi}}}g_0(z),\label{56}
\end{align}
where $g_0(z)$ is a polynomial such that deg$(g_0(z))\leq(s+t-1)(\Gamma_\Psi-\gamma_\Psi)$.\\
On differentiating (\ref{56}), we have
\begin{align} 
\Psi^{(p)}(z)&=A^{\gamma_\Psi} \frac{\prod\limits_{i=1}^{s}{(z-{a_{i}})^{(m_i+1)\gamma_\Psi-\Gamma_\Psi-p}}}{\prod\limits_{j=1}^{t}{(z-{b_{j}})^{(l_j-1)\gamma_\Psi+\Gamma_\Psi+p}}}g_1(z),\label{57} \\
\Psi^{(p+1)}(z)&=A^{\gamma_\Psi} \frac{\prod\limits_{i=1}^{s}{(z-{a_{i}})^{(m_i+1)\gamma_\Psi-\Gamma_\Psi-p-1}}}{\prod\limits_{j=1}^{t}{(z-{b_{j}})^{(l_j-1)\gamma_\Psi+\Gamma_\Psi+p+1}}}g_2(z),\label{58}
\end{align}
where $g_1(z)$ and $g_2(z)$ are the polynomials such that deg$(g_1(z))\leq (s+t-1)(\Gamma_\Psi-\gamma_\Psi+p)$ and deg$(g_2(z))\leq (s+t-1)(\Gamma_\Psi-\gamma_\Psi+p+1)$.\\
\textbf{Case 2.1} Suppose that $\Psi(z)-\alpha(z)$ has exactly one zero say $z_0$. We set
\begin{align} \Psi(z)=\alpha(z)+\frac{(z-z_0)^l}{\prod\limits_{j=1}^{t}{(z-{b_{j}})^{(l_j-1)\gamma_\Psi+\Gamma_\Psi}}},\label{59}
\end{align}
where $l$ is a positive integer.\\
On differentiating (\ref{59}), we have
\begin{align}
\Psi^{(p)}(z)&= B + \frac{(z-z_0)^{l-p}}{\prod\limits_{j=1}^{t}{(z-{b_{j}})^{(l_j-1)\gamma_\Psi+\Gamma_\Psi+p}}}g_3(z),\label{60} \\
\Psi^{(p+1)}(z)&=\frac{(z-z_0)^{l-p-1}}{\prod\limits_{j=1}^{t}{(z-{b_{j}})^{(l_j-1)\gamma_\Psi+\Gamma_\Psi+p+1}}}g_4(z),\label{61}
\end{align}
where $B$ is a non zero constant and $g_3(z)$, $g_4(z)$ are the polynomials such that deg$(g_3(z))\leq pt$ and deg$(g_4(z))\leq (p+1)t$.
On comparing (\ref{57}) and (\ref{60}), we conclude that $z_0\neq a_i(i=1,2,...,s)$.\\
\textbf{Case 2.2.1} $l\neq (N-t)\gamma_\Psi + t\Gamma_\Psi+p$\\
From (\ref{56}) and (\ref{59}), we deduce
\begin{align*} (N-t)\gamma_\Psi+t\Gamma_\Psi &\leq (M+s)\gamma_\Psi-s\Gamma_\Psi +\text{deg}(g_0(z)) \\
                                             &\leq (M+s)\gamma_\Psi-s\Gamma_\Psi +(s+t-1)(\Gamma_\Psi-\gamma_\Psi)\\
                                     \Rightarrow N&<M.
																		\end{align*}
Also, from (\ref{58}) and (\ref{61}), we have
\begin{align*} (&M+s)\gamma_\Psi -s\Gamma_\Psi -(p+1)s\leq\text{deg} (g_4(z))\leq (p+1)t\leq N \\
      \Rightarrow &M\gamma_\Psi \leq s(\Gamma_\Psi -\gamma_\Psi+p+1)+N \\
			\Rightarrow &M<\left( \frac{\Gamma_\Psi -\gamma_\Psi+p+1}{(k+p)\gamma_\Psi}+\frac{1}{\gamma_\Psi}\right)M \\
			\Rightarrow &M<M,
			\end{align*}
which is absurd.\\
\textbf{Case 2.2.2} $l= (N-t)\gamma_\Psi + t\Gamma_\Psi+p$\\
When $M>N$, then by proceeding similar way as in case 2.2.1, we get a contradiction.\\
When $M\leq N$, then from (\ref{58}) and (\ref{61}), we have  
\begin{align*} &l-p-1\leq\text{deg} (g_2(z))\leq (s+t-1)(\Gamma_\Psi -\gamma_\Psi+p+1)\\
          \Rightarrow &l\leq (s+t-1)(\Gamma_\Psi -\gamma_\Psi+p+1)+p+1 \\
					\Rightarrow &(N-t)\gamma_\Psi + t\Gamma_\Psi+p \leq (s+t-1)(\Gamma_\Psi -\gamma_\Psi+p+1)+p+1 \\
					\Rightarrow &N< M\frac{\Gamma_\Psi -\gamma_\Psi+p+1}{(k+p)\gamma_\Psi}+ N\frac{1}{\gamma_\Psi} \\
          \Rightarrow &N< \left( \frac{\Gamma_\Psi -\gamma_\Psi+p+1}{(k+p)\gamma_\Psi}+\frac{1}{\gamma_\Psi}\right)N\\
					\Rightarrow &N<N,
\end{align*}
which is absurd.\\
\textbf{Case 2.2} Suppose that $\Psi(z)-\alpha(z)$ has no zero, then by proceeding the same manner as in case 2.1, we get a contradiction. Hence the Lemma follows.
\end{proof}

\begin{lemma}\label{lemma2} Let $n_0,n_1,n_2,...,n_k,k$ be the non negative integers with $n_0\geq2$, $\sum\limits_{j=1}^{k}n_j\geq 1$ and $k\geq1$. Let $f$ be a transcendental meromorphic function such that $f$ has only zeros of multiplicity at least $k+1$. Then $f^{n_0} (f')^{n_1} \cdots(f^{(k)} )^{n_k}-a(z)$ has infinitely many zeros for any small function $a(z)(\not\equiv0,\infty)$ of $f$.
\end{lemma}
\begin{proof} Let $\Psi:=f^{n_0} (f')^{n_1} \cdots(f^{(k)} )^{n_k}$. Suppose on contrary that $\Psi(z)-a(z)$ has only finitely many zeros. Then by second fundamental theorem of Nevanlinna for three small functions, we have
\begin{align} T(r,\Psi)&\leq \overline N(r,\Psi) + \overline N\left(r,\frac{1}{\Psi}\right) + \overline N\left(r,\frac{1}{\Psi - a(z)}\right)\notag\\
&\leq \overline N(r,f) + \overline N\left( r, \frac{1}{\Psi}\right) + S(r,\Psi).\label{50}
\end{align}
Also, from \cite[Theorem 1.1, see (2.2)]{charak-1}, we have
\begin{align} \overline N\left( r, \frac{1}{\Psi}\right)\leq \left( 1+\sum\limits_{j=1}^{k}\right) \overline N\left(r,\frac{1}{f}\right) + \sum\limits_{j=1}^{k} j \overline N(r,f) + S(r,\Psi),\label{51}
\end{align}
where the summation runs over the terms $f^{(j)}(1\leq j\leq k)$ in the monomial $\Psi$ of $f$.\\
Since the zeros of $f$ has multiplicity at least $k+1$, we have
\begin{align} \overline N \left(r,\frac{1}{f}\right)&=\overline N_{(k+1}\left(r,\frac{1}{f}\right) \notag \\
                         & \leq \frac{1}{kn_0+\gamma-1} \left(N\left( r,\frac{1}{\Psi}\right)-\overline N\left(r,\frac{1}{\Psi}\right)\right). \label{52}\end{align}

Also, one can see that 
\begin{align} \overline N(r,f)\leq \frac{1}{\Gamma_\Psi} N(r,\Psi).\label{53}
\end{align}
On combining (\ref{50}), (\ref{51}), (\ref{52}) and (\ref{53}), we get 
\begin{align*} T(r,\Psi) \leq \frac{1+\sum\limits_{j=1}^{k} j}{kn_0+\gamma+\sum\limits_{j=1}^{k} j}N\left(r,\frac{1}{\Psi}\right) &+\left( 1+ \frac{\sum\limits_{j=1}^{k} j (kn_0+\gamma-1)}{kn_0+\gamma+\sum\limits_{j=1}^{k} j}\right) \overline N(r,\Psi) + S(r,\Psi)\\
\Rightarrow \left( 1-\frac{1}{\Gamma_{\Psi}} - \frac{1+\sum\limits_{j=1}^{k} j (kn_0+\gamma-1)}{(kn_0+\gamma+\sum\limits_{j=1}^{k} j)\Gamma_{\Psi}}\right)&T(r,\Psi)\leq S(r,\Psi) \\
\Rightarrow &T(r,\Psi)\leq S(r,\Psi).
\end{align*}

Thus, by using the inequality \cite{singh-1}
$$T(r,f)+S(r,f)\leq CT(r,\Psi)+S(r,\Psi),$$
where $C$ is a constant and $S(r,f)=S(r,\Psi)$, we get a contradiction. Hence the Lemma follows.
\end{proof}

Since normality is local property, we shall always assume that $D=\mD$, the open unit disk and hence the point at which the family is assumed to be not normal is the origin.
		
\begin{proof} [\textbf{Proof of Theorem 1.2.}]
Suppose that $h_1$ and $h_2$ are two distinct zeros of $P(z_0,w)-\alpha(z_0)$ for any $z_0\in D$. Suppose on contrary that $\mathcal{F}$ is not normal at the origin. Then by Zalcman's lemma \cite[p.216]{zalcman-1}, we can find a sequence $\left\{ f_j \right\}$ in $\mathcal{F}$, a sequence $\left\{ z_j\right\}$ of complex numbers with $z_j\rightarrow 0$ and a sequence $\left\{\rho_j\right\}$ of positive real numbers with $\rho_j \rightarrow 0$ such that
$$g_j (\zeta) =\rho_j ^{-\beta} f_j (z_j +\rho_j \zeta)$$
converges locally uniformly with respect to the spherical metric to a non constant meromorphic function $g(\zeta)$ on $\mC$ having bounded spherical derivative.\\
We take $\beta=\frac{\Gamma_M}{\gamma_M} -1.$\\
Hence, 
$$P_w \circ M[f_j](z_j+\rho_j\zeta)-\alpha(z_j+\rho_j\zeta)\rightarrow P_w \circ M[g](\zeta)-\alpha(0).$$
We claim that $P_w \circ M[g](\zeta)-\alpha(0)$ has at least two distinct zeros. We distinguish the following cases:\\
\textbf{Case I}: Suppose that $P_w \circ M[g](\zeta)-\alpha(0)$ has exactly one zero say $a$. That is,
$P_w \circ M[g](a)=\alpha(0)$ and
$P_w \circ M[g](\zeta)\neq\alpha(0)$ for $\zeta\neq a$. Without loss of generality we may assume that $ M[g](a)=h_1$. Then $M[g](\zeta)\neq h_2$ and $M[g](\zeta)\neq h_1$ for $\zeta\neq a$. If $g(\zeta)$ is a transcendental meromorphic function, then we have 
\begin{align}
\Gamma_M \overline{N}(r,g)&\leq N(r,M[g])\notag\\
&\leq T(r,M[g])\notag\\
&\leq \overline{N}(r,M[g])+\overline{N}\left(r,\frac{1}{M[g]-h_1}\right)+\overline{N}\left(r,\frac{1}{M[g]-h_2}\right)+S(r,M[g])\notag\\
&\leq \overline{N}(r,g) +S(r,g)\notag\\
\Rightarrow (\Gamma_M -1)\overline{N}(r,g)&\leq S(r,g)\notag\\
\Rightarrow \overline{N}(r,g)&=S(r,g).\label{2}
\end{align}
Since the zeros of $g$ has multiplicity at least $k$, we have
\begin{equation}\overline N\left(r,\frac{1}{g}\right)\leq\frac{1}{k}N\left(r,\frac{1}{g}\right) \label{3}
\end{equation}
Now, 
\begin{align*}
\gamma_M T(r,g)&=T(r,g^{\gamma_M})\\
&=T\left(r,\frac{1}{g^{\gamma_M}}\right)+O(1)\\
&\leq T\left(r,\frac{M[g]}{g^{\gamma_M}}\right)+T\left(r,\frac{1}{M[g]})\right)+S(r,g)\\
&\leq T\left(r,\frac{g^{n_0} (g')^{n_1}\cdots (g^{(k)} )^{n_k}}{g^{n_0} g^{n_1}\cdots g^{n_k}}\right)+S(r,g)\\
&=m\left(r,\frac{(g')^{n_1}\cdots (g^{(k)} )^{n_k}}{ g^{n_1}\cdots g^{n_k}}\right)+N\left(r,\frac{(g')^{n_1}\cdots (g^{(k)} )^{n_k}}{ g^{n_1}\cdots g^{n_k}}\right)+S(r,g)\\
&\leq\sum\limits_{i=1}^{k}{\left[ m\left(r,\left(\frac{g^{(i)}}{g}\right)^{n_i}\right) + n_i N\left(r,\left(\frac{g^{(i)}}{g}\right)\right) \right]}+S(r,g)\\
&\leq\sum\limits_{i=1}^{k}{i{{n}_{i}}\left[ \overline{N}(r,g)+\overline{N}\left(r,\frac{1}{g}\right) \right]}+S(r,g).
\end{align*}
That is,
\begin{equation}\gamma_M T(r,g)\leq(\Gamma_M -\gamma_M)\left[ \overline{N}(r,g)+\overline{N}\left(r,\frac{1}{g}\right) \right]+S(r,g).\label{4}
\end{equation}
Substituting (\ref{2}) and (\ref{3}) in (\ref{4}), we get
\begin{align*}\gamma_M T(r,g)&\leq\frac{\Gamma_M -\gamma_M}{k} N\left(r,\frac{1}{g}\right) +S(r,g)\\
                             &\leq\frac{\Gamma_M -\gamma_M}{k} T\left(r,\frac{1}{g}\right) +S(r,g) \\
								            &\leq\frac{\Gamma_M -\gamma_M}{k} T(r,g) +S(r,g) \\
               \Rightarrow\frac{(k+1)\gamma_M -\Gamma_M}{k}T(r,g)&\leq S(r,g) \\
							 \Rightarrow T(r,g)&=S(r,g),
\end{align*}
which is a contradiction. If $g(\zeta)$ is a non polynomial rational function, then we consider the following subcases:\\
\textbf{Subcase I}: $h_2 \neq0$\\
 For a non-zero $h_2$, $M[g](\zeta)-h_2$ has at least one finite zero, see \cite[Lemma 2.6, p.6]{zhang-1}. This is a contradiction to the fact that $M[g](\zeta)\neq h_2$.\\
\textbf{Subcase II}: $h_2 =0$ \\
Since $M[g](\zeta)\neq h_2=0$, so we deduce
                 \begin{equation} g(\zeta)=\frac{A}{(\zeta-b_1)^{l}},\end{equation}
								where $A\neq0$ is a constant and $l\geq1$.\\
Now, 
\begin{align*}
M[g](\zeta)=\frac{A^{\gamma_M}}{(\zeta-b_1)^{(l -1)\gamma_M +\Gamma_M}}.
\end{align*}
So,
\begin{align}\label{5}
(M[g](\zeta))'=\frac{B}{(\zeta-b_1)^{(l -1)\gamma_M +\Gamma_M+1}}.
\end{align}
Since, $M[g](a)=h_1$, we have
\begin{align*} 
M[g](\zeta)-h_1 &=\frac{A^{\gamma_M}}{(\zeta-b_1)^{(l -1)\gamma_M +\Gamma_M}}-h_1\\
&=C\frac{(\zeta -a)^{(l -1)\gamma_M +\Gamma_M}}{(\zeta-b_1)^{(l -1)\gamma_M +\Gamma_M}},
\end{align*}
where $C\neq0$ is a constant.\\
Therefore, 
\begin{align}\label{6}
(M[g](\zeta))'&=(M[g](\zeta)-h_1)'\\
&=\frac{(\zeta -a)^{(l -1)\gamma_M +\Gamma_M-1}}{(\zeta-b_1)^{(l -1)\gamma_M +\Gamma_M+1}}R(\zeta),\notag
\end{align}
where $R(\zeta)$ is a polynomial.
On comparing (\ref{5}) and (\ref{6}), we have
\begin{align*}
(l -1)\gamma_M + \Gamma_M -1 + \text{deg}(R(\zeta))=0\\
\Rightarrow (l -1)\gamma_M + \Gamma_M -1 \leq 0 \\
\Rightarrow (k-1)\gamma_M + \Gamma_M -1 \leq 0,
\end{align*}
which is a contradiction.\\
 If $g(\zeta )$ is a non constant polynomial, then the polynomial $M[g](\zeta)$ cannot avoid $h_1$ and $h_2$ and so $P_w \circ M[g](\zeta)-\alpha(0)$ has at least two distinct zeros, which is a contradiction.\\
\textbf{Case II}: Suppose that $P_w \circ M[g](\zeta)-\alpha(0)$ has no zero. Then $M[g](\zeta)\neq h_1,h_2$. By proceeding the same way as in case I, we get a contradiction.\\
Thus the claim hold and hence there exist $\zeta_0$ and $\tilde{\zeta} _0$ such that 
$$P_w \circ M[g](\zeta_0)-\alpha(0)=0$$ and
$$P_w \circ M[g](\tilde{\zeta}_0)-\alpha(0)=0.$$
Consider $N_1 =\left\{\zeta\in\mathbb{C} : |\zeta-\zeta_0|<\delta\right\}$ and $N_2 =\left\{\zeta\in\mathbb{C} : |\zeta-\tilde{\zeta} _0|<\delta\right\}$, where $\delta(>0)$ is small enough number such that $N_1 \cap N_2 =\phi$ and $P_w \circ M[g](\zeta)-\alpha(0)$ has no other zeros in $N_1 \cup N_2$ except for $\zeta_0$ and $\tilde{\zeta} _0$. By Hurwitz theorem, there exist points $\zeta_j \rightarrow \zeta_0$ and $\tilde{\zeta}_j \rightarrow \tilde{\zeta}_0$ such that, for sufficiently large $j$, we have
$$P_w \circ M[f_j](z_j+\rho_j\zeta_j)-\alpha(z_j+\rho_j\zeta_j)=0,$$ and
$$P_w \circ M[f_j](z_j+\rho_j\tilde{\zeta}_j)-\alpha(z_j+\rho_j\tilde{\zeta}_j)=0.$$
Since, $P_w \circ M[f](z)$ and $P_w \circ M[g](z)$ share $\alpha(z)$ IM in $D$, we see for any integer $i$
$$P_w \circ M[f_i](z_j+\rho_j\zeta_j)-\alpha(z_j+\rho_j\zeta_j)=0,$$ and
$$P_w \circ M[f_i](z_j+\rho_j\tilde{\zeta}_j)-\alpha(z_j+\rho_j\tilde{\zeta}_j)=0.$$
For a fixed $i$, taking $j\rightarrow \infty$, we have 
$$P_w \circ M[f_i](0)-\alpha(0)=0.$$
Since, the zeros of $P_w \circ M[f_i](\zeta)-\alpha(\zeta)$ have no accumulation point except for finitely many $f_i$, so for sufficiently large $j$, we have
\begin{align*}&z_j+\rho_j \zeta_j=0,~~~~~~~~~z_j+\rho_j \tilde{\zeta}_j=0\\
&\text{or}~~~~~~~~\zeta_j=\frac{z_j}{\rho_j},~~~~~~~~~\tilde{\zeta}_j=\frac{z_j}{\rho_j}
\end{align*}
which is a contradiction to the fact that $\zeta_j\in N_1$ , $\tilde{\zeta}_j\in N_2$ and $N_1 \cap N_2 =\phi$. Hence $\mathcal{F}$ is normal in $D$.
\end{proof}

\begin{proof}[\textbf{Proof of Corollary 1.5.}]
Let $S=\left\{a_l(z):l=1,2,\cdots,n\right\}$ be the set of holomorphic functions on a domain $D$ with $n\geq2$. Then consider $P(z,w):=(w-a_1(z))(w-a_2(z))\cdots(w-a_n(z))$  and $\alpha(z)\equiv 0$. Clearly, for every $f_i,f_j\in \mathcal{F}$, $P_w \circ M[f_i](z)$ and $P_w \circ M[f_j](z)$ share $\alpha(z)$ IM. Hence, by applying Theorem \ref{theorem1}, $F$ is normal in $D$.
\end{proof}

\begin{proof} [\textbf{Proof of Theorem 1.6.}]
Suppose on contrary that $\mathcal{G}$ is not normal at the origin. We consider the following cases:\\
\textbf{Case I}: $\alpha(0)\neq 0$\\
By Zalcman's lemma \cite[p.216]{zalcman-1}, we can find a sequence $\left\{ f_j \right\}$ in $\mathcal{G}$, a sequence $\left\{ z_j\right\}$ of complex numbers with $z_j\rightarrow 0$ and a sequence $\left\{\rho_j\right\}$ of positive real numbers with $\rho_j \rightarrow 0$ such that
$$g_j (\zeta) =\rho_j ^{-\beta} f_j (z_j +\rho_j \zeta)$$
converges locally uniformly with respect to the spherical metric to a non constant meromorphic function $g(\zeta)$ on $\mC$ having bounded spherical derivative.\\
We take $\beta=\frac{\Gamma_M}{\gamma_M} -1.$\\
Hence $$M[f_j](z_j+\rho_j\zeta)-\alpha(z_j+\rho_j\zeta) \rightarrow M[g](\zeta)-\alpha(0).$$
Clearly, $M[g](\zeta)\not\equiv \alpha(0)$. Thus by Lemma \ref{lemma1} and Lemma \ref{lemma2}, $M[g](\zeta)-\alpha(0)$ has at least two distinct zeros. Suppose that $\zeta_0$ and $\tilde{\zeta} _0$ be two distinct zeros of $M[g](\zeta)-\alpha(0)$.
Consider $N_1 =\left\{\zeta\in\mathbb{C} : |\zeta-\zeta_0|<\delta\right\}$ and $N_2 =\left\{\zeta\in\mathbb{C} : |\zeta-\tilde{\zeta} _0|<\delta\right\}$, where $\delta(>0)$ is small enough number such that $N_1 \cap N_2 =\phi$ and $M[g](\zeta)-\alpha(0)$ has no other zeros in $N_1 \cup N_2$ except for $\zeta_0$ and $\tilde{\zeta} _0$. By Hurwitz theorem, there exist points $\zeta_j \rightarrow \zeta_0$ and $\tilde{\zeta}_j \rightarrow \tilde{\zeta}_0$ such that, for sufficiently large $j$, we have
$$M[f_j](z_j+\rho_j\zeta_j)-\alpha(z_j+\rho_j\zeta_j)=0,$$
$$M[f_j](z_j+\rho_j\tilde{\zeta}_j)-\alpha(z_j+\rho_j\tilde{\zeta}_j)=0.$$
Since, $M[f](z)$ and $M[g](z)$ share $\alpha(z)$ IM in $D$, we see for any integer $i$
$$M[f_i](z_j+\rho_j\zeta_j)-\alpha(z_j+\rho_j\zeta_j)=0,$$
$$M[f_i](z_j+\rho_j\tilde{\zeta}_j)-\alpha(z_j+\rho_j\tilde{\zeta}_j)=0.$$
For a fixed $i$, taking $j\rightarrow \infty$, we have 
$$M[f_i](0)-\alpha(0)=0.$$
Since, the zeros of $M[f_i](\zeta)-\alpha(0)$ have no accumulation point, so for sufficiently large $j$, we have
\begin{align*}&z_j+\rho_j \zeta_j=0,~~~~~~~~~z_j+\rho_j \tilde{\zeta}_j=0\\
&\text{or}~~~~~~~~\zeta_j=\frac{z_j}{\rho_j},~~~~~~~~~\tilde{\zeta}_j=\frac{z_j}{\rho_j}
\end{align*}
which is a contradiction to the fact that $\zeta_j\in N_1$ , $\tilde{\zeta}_j\in N_2$ and $N_1 \cap N_2 =\phi$.\\
\textbf{Case II}: $\alpha(0)= 0$\\
 We assume $\alpha(z)=z^p \alpha_1(z)$, where $p$ is a positive integer and $\alpha_1(0)\neq0$. We may take $\alpha_1(0)=1$. By Zalcman's lemma \cite[p.216]{zalcman-1}, we can find a sequence $\left\{ f_j \right\}$ in $\mathcal{G}$, a sequence $\left\{ z_j\right\}$ of complex numbers with $z_j\rightarrow z_0$ and a sequence $\left\{\rho_j\right\}$ of positive real numbers with $\rho_j \rightarrow 0$ such that
$$g_j (\zeta) =\rho_j ^{-\eta} f_j (z_j +\rho_j \zeta)$$
converges locally uniformly with respect to the spherical metric to a non constant meromorphic function $g(\zeta)$ on $\mC$ having bounded spherical derivative.\\
We take $\eta=\frac{\Gamma_M-\gamma_M+p}{\gamma_M}$\\
Now, we cosider the following subcases:\\
\textbf{Subcase I}: Suppose there exist a subsequence of $\frac{z_j}{\rho_j}$, for simplicity we take $\frac{z_j}{\rho_j}$ itself, such that $\frac{z_j}{\rho_j} \rightarrow \infty $ as $j \rightarrow \infty$.\\
Let $\Psi:=\left\{H_j(\zeta)=z_j ^{-\eta} f_j(z_j + z_j\zeta), \forall \left\{f_j\right\}\subset \mathcal{G}\right\}$. Thus by the given condition, we have 
 \begin{align*}H_j^{n_0}(\zeta) (H'_j)^{n_1}(\zeta) \cdots(H_j^{(k)} )^{n_k}(\zeta)&=(1+\zeta)^p \alpha_1(z_j + z_j\zeta)\\
 &\Leftrightarrow G_j^{n_0}(\zeta) (G'_j)^{n_1}(\zeta) \cdots(G_j^{(k)} )^{n_k}(\zeta)=(1+\zeta)^p \alpha_1(z_j + z_j\zeta),
\end{align*}
where $G_j(\zeta) =z_j ^{-\eta} g_j(z_j + z_j\zeta)$ for every $\left\{g_j\right\}\subset \mathcal{G}$. Thus, by case I, $\Psi$ is normal in $\mD$ and hence there exist a subsequence of $\left\{H_j\right\}$ in  $\Psi$, we may take $\left\{H_j\right\}$ itself, such that $H_j$ converges spherically uniformly to $H$(say) on \mD. We claim $H(0)=0$. Suppose $H(0)\neq0$, then we have
$$g_j(\zeta)=\rho^{-\eta} f_j(z_j+ \rho_j \zeta)=\left(\frac{z_j}{\rho_j}\right)^{\eta} H_j \left(\frac{\rho_j}{z_j}\zeta\right)$$ converges locally uniformly to $\infty$ on $\mathbb{C}$. This implies that $g(\zeta)\equiv \infty$, which is a contradiction. Thus the claim hold and hence $H'(0)\neq\infty$. Now, for any $\zeta\in\mathbb{C}$, we have 
\begin{align*}g'_j (\zeta)&= \rho_j^{-\eta+1} f'_j (z_j + \rho_j\zeta) \\
                          &= \left(\frac{\rho_j}{z_j}\right)^{-\eta +1} H'_j \left( \frac{\rho_j}{z_j}\zeta\right) \overset{\chi }{\mathop{\to }}\, 0
\end{align*}
on $\mathbb{C}$. Thus $g' (\zeta)\equiv 0$. This implies that $g$ is a constant, which is a contradiction.\\
\textbf{Subcase II}: Suppose there exist a subsequence of $\frac{z_j}{\rho_j}$, for simplicity we take $\frac{z_j}{\rho_j}$ itself, such that $\frac{z_j}{\rho_j} \rightarrow c $ as $j \rightarrow \infty$, where c is a finite number.\\
Let $\phi_j(\zeta)=\rho_j ^{-\eta}f_j(\rho_j \zeta)$. Then $$\phi_j(\zeta)=g_j\left(\zeta-\frac{z_j}{\rho_j}\right)\rightarrow g(\zeta-c):=\phi(\zeta).$$
Thus \begin{align*} \phi_j^{n_0}(\zeta) (\phi_j')^{n_1}(\zeta) &\cdots(\phi_j^{(k)} )^{n_k}(\zeta)-\zeta^p \\
                           &=g^{n_0}\left(\zeta-\frac{z_j}{\rho_j}\right) (g')^{n_1}\left(\zeta-\frac{z_j}{\rho_j}\right) \cdots(g^{(k)} )^{n_k}\left(\zeta-\frac{z_j}{\rho_j}\right)- \left(\zeta-\frac{z_j}{\rho_j}\right)^p\\
													&\rightarrow g^{n_0}(\zeta-c) (g')^{n_1}(\zeta-c) \cdots(g^{(k)} )^{n_k}(\zeta-c)- (\zeta-c)^p\\
													&:=\phi^{n_0}(\zeta) (\phi')^{n_1}(\zeta)\cdots(\phi^{(k)} )^{n_k}(\zeta)-\zeta^p
\end{align*}
Clearly, $\phi^{n_0}(\zeta) (\phi')^{n_1}(\zeta)\cdots(\phi^{(k)} )^{n_k}(\zeta)-\zeta^p \not\equiv0$. Thus by Lemma \ref{lemma1} and Lemma \ref{lemma2}, $\phi^{n_0}(\zeta) (\phi')^{n_1}(\zeta)\cdots(\phi^{(k)} )^{n_k}(\zeta)-\zeta^p$ has at least two distinct zeros. Thus by proceeding the same way as in Case I, we get a contradiction.
Hence $\mathcal{G}$ is normal in $D$.
\end{proof}

\begin{proof} [\textbf{Proof of Theorem 1.7.}]
Suppose  $\mathcal{F}$ is not normal at the origin. Then by Zalcman's lemma \cite[p.216]{zalcman-1}, we can find a sequence $\left\{ f_j \right\}$ in $\mathcal{F}$, a sequence $\left\{ z_j\right\}$ of complex numbers with $z_j\rightarrow 0$ and a sequence $\left\{\rho_j\right\}$ of positive real numbers with $\rho_j \rightarrow 0$ such that
$$g_j (\zeta) =\rho_j ^{-\beta} f_j (z_j +\rho_j \zeta)$$
converges locally uniformly with respect to the spherical metric to a non constant meromorphic function $g(\zeta)$ on $\mC$ having bounded spherical derivative.\\
We take $\beta=\frac{\Gamma_M}{\gamma_M} -1.$\\
Hence, \begin{align}P_w \circ M[f_j](z_j+\rho_j\zeta)-\alpha(z_j+\rho_j\zeta)\rightarrow P_w \circ M[g](\zeta)-\alpha(0).\label{7}\end{align}
Clearly, $P_w \circ M[g](\zeta)\not\equiv\alpha(0)$. We claim that $P_w \circ M[g](\zeta)-\alpha(0)$ has at most one zero. Suppose that $P_w \circ M[g](\zeta)-\alpha(0)$ has two distinct zeros say $\zeta_1$ and $\zeta_2$. Then by (\ref{7}) and Hurwitz theorem, there exist points $\zeta_{j1}\rightarrow \zeta_1 $ and $\zeta_{j2}\rightarrow \zeta_2 $ such that 
$$P_w \circ M[f_j](z_j+\rho_j\zeta_{j1})-\alpha(z_j+\rho_j\zeta_{j1})=0$$ and
$$P_w \circ M[f_j](z_j+\rho_j\zeta_{j2})-\alpha(z_j+\rho_j\zeta_{j2})=0,$$
for sufficiently large $j$. Since, $z_j+\rho_j\zeta_{ji}\rightarrow 0 (i=1,2)$ and $P_w \circ M[f_j](z_j+\rho_j\zeta)-\alpha(z_j+\rho_j\zeta)$ has at most one zero, we get a contradiction. Thus the claim holds. But as shown in the proof of Theorem \ref{theorem1}, $P_w \circ M[g](\zeta)-\alpha(0)$ has at least two distinct zeros and this contradicts our claim. Hence $\mathcal{F}$ is normal in $D$.
\end{proof}

\begin{proof} [\textbf{Proof of Theorem 1.8.}]
Suppose $\mathcal{F}$ is not normal at the origin. Then as in the proof of the Theorem \ref{theorem3}, we find that 
 $$P_w \circ M[f_j](z_j+\rho_j\zeta)-\alpha(z_j+\rho_j\zeta)\rightarrow P_w \circ M[g](\zeta)-\alpha(0).$$
Clearly, $P_w \circ M[g](\zeta)\not\equiv\alpha(0)$. As shown in the proof of Theorem \ref{theorem1}, $P_w \circ M[g](\zeta)-\alpha(0)$ has at least one zero say $\zeta_0$ and hence $g(\zeta_0)\neq\infty$. By Hurwitz theorem, for sufficiently large $j$, there exist a sequence $\zeta_j\rightarrow\zeta_0$ as $j\rightarrow \infty$ such that
$$P_w \circ M[f_j](z_j+\rho_j\zeta_j)-\alpha(z_j+\rho_j\zeta_j)=0.$$
Thus by the given condition, $P_w \circ M[f](z)=\alpha(z)$ implies $\left|f(z)\right|\geq M$, we have
\begin{align*} \left|g_j(\zeta_j)\right|&=\rho^{-\beta}_j\left|f_j(z_j+\rho_j\zeta_j)\right| \\
                          &\geq \rho^{-\beta}_jM.
 \end{align*}
Since $g(\zeta_0)\neq\infty$ in some neighborhood of $\zeta_0$ say $N_{\zeta_0}$, it follows that for sufficiently large values of $j$, $g_j(\zeta)$ converges uniformly to $g(\zeta)$ in $N_{\zeta_0}$. Thus for $\epsilon>0$ and for every $\zeta\in N_{\zeta_0}$, we have 
$$\left|g_j(\zeta)-g(\zeta)\right|<\epsilon.$$
Therefore, for sufficiently large values of $j$, we have
$$\left|g(\zeta_j)\right|\geq\left|g_j(\zeta_j)\right|-\left|g(\zeta_j)-g_j(\zeta_j)\right|\geq\rho^{-\beta}_jM-\epsilon,$$
which is a contradiction to the fact that $\zeta_0$ is not a pole of $g(\zeta)$. Hence $\mathcal{F}$ is normal in $D$.
\end{proof}

\bibliographystyle{amsplain}

\end {document}